\documentclass[a4paper,12pt]{smfart}
  \usepackage[T1]{fontenc}
  \usepackage[utf8]{inputenc}
  \usepackage{smfthm,amsmath}
  \usepackage{amssymb}
  \usepackage{mathrsfs}
  \usepackage[french]{babel}
  \usepackage{graphicx,color}

\newcommand{\project}{\mathrm{P}}
\newcommand{\bK}{\mathbb{K}}
\newcommand{\bR}{\mathbb{R}}

\newcommand{\rep}{\bullet}
\newcommand{\plan}{\mathrm{P}}

\newtheorem{ques}{Question}

\title{Polarités définies par un triangle}
\author{Benoît Kloeckner}

\begin{document}

\begin{abstract}
Une polarité d'un plan projectif est une application, souvent
involutive, envoyant un point
générique sur une droite générique et réciproquement.
La polarité la plus classique est la polarité par rapport à une conique, mais
d'autres existent : la polarité harmonique par rapport à un triangle, les
polarités par rapport à une courbe algébrique de degré supérieur, la polarité
par rapport à un convexe.

Dans cet article nous introduisons une notion de polarité par rapport à un
triangle du plan projectif, motivée par une question sur la dualité des 
repères projectifs. Nous montrons que les quatre polarités évoquées qui
peuvent s'appliquer à un triangle coïncident dans ce cas.
Ce résultat fournit un prétexte à passer en revue de jolis concepts
de géométrie projective, d'algèbre linéaire et de géométrie convexe.
\end{abstract}

\maketitle

\section{Introduction}

On va rappeler l'ensemble des notions de géométrie projective
utilisées ici, mais le lecteur souhaitant plus de détail peut
par exemple consulter \cite{K:livre}.

On se place sur un plan projectif $\plan=\project(V)$, c'est-à-dire
sur l'ensemble des droites vectorielles d'un espace $V$ de dimension $3$ sur
un corps $\bK$. On notera $\project(\vec v)\in\plan$ la droite définie
par un vecteur $\vec v\in V$ non nul.

On fixe un \emph{triangle} $\Delta$, c'est-à-dire trois points 
non alignés $p_1,p_2,p_3\in\plan$. Les \emph{côtés} de $\Delta$
sont les droites $D_1=(p_2p_3)$, $D_2=(p_1p_3)$ et $D_3=(p_1p_2)$.

On dira qu'un point $p$, respectivement une droite $D$,
est \emph{générique} s'il n'appartient à aucun côté de $\Delta$, respectivement
si elle ne passe par aucun des sommets de $\Delta$.

Présentons rapidement les polarités usuelles que nous allons étudier. Elle
seront définies précisément plus loin.

Tout d'abord, la \emph{polarité harmonique} est définie par une construction
explicite à la règle~; elle généralise la notion de quatrième harmonique d'un
triplet dans une droite projective. On notera $p^\#$
la droite harmonique d'un point générique $p$ par rapport à $\Delta$, 
et $D^\#$ le point harmonique d'une droite générique $D$.

Ensuite, de la même façon qu'une conique définit une polarité à l'aide de sa forme
quadratique associée, une cubique définit (si $\bK$ est
de caractéristique différente de $2$ et $3$) deux «~polarités~»,
la première envoyant un point sur une conique et la deuxième un point sur une
droite. Les trois côtés de $\Delta$ définissent une cubique, on notera
$p^\perp$ la seconde polaire d'un point générique $p$ par rapport à cette cubique
et $D^\perp$ le point dont la droite générique $D$ est la seconde 
polaire.

Enfin, la notion affine de centre de gravité permet de définir quand
$\bK=\bR$ une polarité par rapport à un convexe ouvert propre, qu'on qualifiera de 
«~polarité convexe~». Elle n'est définie que pour les points appartenant
à l'intérieur du convexe, et pour les droites qui n'en rencontrent pas
l'adhérence. L'ensemble des points génériques est découpé par
les côtés de $\Delta$ en quatre composantes connexes, toutes convexes et 
triangulaires. On notera $p^\circ$ la polaire convexe du point générique $p$
par rapport à la composante à laquelle il appartient, et $D^\circ$ le polaire
convexe de la droite générique $D$ par rapport à l'unique composante qu'elle
ne rencontre pas.

Introduisons maintenant une nouvelle notion de polarité~: un point générique
$p$ définit un repère projectif $(p_1,p_2,p_3,p)$, auquel est associé un repère
projectif de $\plan^*=\project(V^*)$, constitué de quatre 
points dont les droites duales sont des droites de 
$\plan$. Il est facile de vérifier que les trois premières sont les côtés
de $\Delta$, on note $p^\rep$ la quatrième et on l'appelle la \emph{polaire repère}
de $p$ par rapport à $\Delta$. La construction réciproque permet de définir
le point polaire repère $D^\rep$ d'une droite générique $D$.

Ces quatre polarités par rapport à un triangle sont involutives,
et tout résultat sur les points se traduit donc sur les droites.
Le théorème suivant est élémentaire mais servira de fil rouge, connectant les différents
sujets que nous allons passer en revue.
\begin{theo}
Chaque fois que $\bK$ permet de définir les termes impliqués,
pour tout point générique $p$ on a :
\[ p^\rep=p^\#=p^\perp=p^\circ. \]
\end{theo}
En particulier, comme la polarité harmonique se construit à la règle, ce
résultat montre qu'il est possible de construire à la règle le repère
dual d'un repère projectif, à partir de la donnée des droites duales
des points du repère original. Cette question simple, mais
qui semble n'avoir pas été précédemment investiguée, est la motivation initiale
de ce travail.

Chacune des quatre prochaines sections est dévolue à une des polarités, 
et on prouve
au fur et à mesure qu'elles coïncident. On s'est en général placé
en dimension $2$ pour simplifier les notations et les figures, mais
la dernière section explique
brièvement comment ceci se généralise en dimension supérieure.

\section{Polarité repère}

Dans cette section, on présente quelques rappels élémentaires
sur les repères projectifs et on introduit la notion de repère dual et la notion
de polarité associée. Bien qu'elle soit très naturelle, 
je n'en connais pas d'occurrence dans la littérature.

\subsection{Repères et coordonnées}

À une base $\vec{\mathscr{B}}=(\vec b_1,\vec b_2,\vec b_3)$
de $V$ on associe les points
$r_i=\project(\vec b_i)$ pour $1\leqslant i\leqslant 3$ et
$r_+=\project(\vec b_1+\vec b_2+\vec b_3)$. Alors les éléments de 
$\mathscr{R}=(r_1,r_2,r_3,r_+)$ sont trois à trois non alignés~;
on dit que $\mathscr{R}$ est un \emph{repère} de $\project(V)$ et que
la base  $\vec{\mathscr{B}}$ est \emph{adaptée} au repère $\mathscr{R}$. 
Remarquons
que toute famille obtenue en enlevant un élément à $\mathscr{R}$ engendre 
$\project(V)$.

Pour tout repère de $\project(V)$ (quatre points trois à trois non alignés,
donc), il existe une base adaptée, unique à  multiplication par un scalaire
près. Notons que pour obtenir cette unicité, le point $r_+$ est absolument
nécessaire~: les bases $(\lambda_i \vec b_i)_i$ définissent les mêmes $r_i$
que $\vec{\mathscr{B}}$, mais seules celles où $\lambda_i$ ne dépend pas de $i$
définissent également le même point $r_+$.

Le groupe projectif $\mathrm{PGL}(V)=\mathrm{GL}(V)/\bK^*$
agit transitivement sur les repères, et les différents éléments d'un
repère sont en particulier indistingables d'un point de vue projectif~; 
toutefois lorsque l'on met en relation
un repère et une base adaptée, le dernier élément a un rôle particulier, 
c'est pourquoi il est ici noté $r_+$ plutôt que $r_4$.

Le repère $\mathscr{R}$ définit des \emph{coordonnées homogènes} sur 
$\project(V)$~:
à un point $p=\project(\vec v)$ où 
$\vec v= x_1 \vec b_1+x_2\vec b_2+x_3\vec b_3\neq\vec 0$
on associe le triplet défini à multiplication près
\[ [x_1:x_2:x_3]:=\left\{ (\lambda x_1,\lambda x_2,\lambda x_3)\middle|
  \lambda\in\bK\right\}. \]
Si l'on veut s'affranchir de cette indéfinition,
on peut normaliser l'une des coordonnées non nulles~: par exemple
si $x_3\neq 0$,
on dit que $(\frac{x_1}{x_3},\frac{x_2}{x_3})$ sont les coordonnées de $p$ 
dans la  \emph{carte affine} d'équation $(x_3\neq 0)$. Un domaine de carte affine
est le complémentaire d'une droite, dite «~à l'infini~», et s'identifie
au plan affine.

\subsection{Repère dual}

L'espace dual de $\plan$ est par définition l'espace projectif 
$\plan^*:=\project(V^*)$ où $V^*$ est l'espace vectoriel dual de $V$.
On notera $a^*\subset\project(V^*)$ la droite duale d'un point $a\in\plan$ 
(c'est-à-dire le projectivé
de l'espace des formes linéaires s'annulant sur $a$ vu comme droite de $V$),
et de façon similaire le point dual d'une droite $A$ est
$A^*=\project(\phi)\in\plan^*$ où $\phi\in \project(V^*)$
est n'importe quelle forme linéaire non nulle s'annulant sur tout
vecteur $\vec v$ tel que $\project(\vec v)\in A$.


La dualité est involutive~: l'isomorphisme linéaire canonique
\begin{eqnarray*}
e :V &\to& V^{**} \\
\vec v &\mapsto& e_{\vec v} : \phi \mapsto \phi(\vec v)
\end{eqnarray*}
passe au quotient en un isomorphisme projectif 
$\project(e):\project(V)\to\project(V^{**})$
qui permet d'identifier le plan et son bidual. Étant donné 
$\alpha\in\project(V^*)$, on notera par exemple
$\project(e)^{-1}(\alpha^*)$ simplement $\alpha^*$. Cette écriture permet
ainsi d'avoir $a^{**}=a$.

\begin{defi}[repère dual]
Étant donné un repère $\mathscr{R}$ de $\project(V)$ et  une base adaptée 
$\vec{\mathscr{B}}$, on considère la base duale $\vec{\mathscr{B}}^*$, et on 
note
$\mathscr{R}^*$ le repère de $\project(V^*)$
qui lui est associé. Comme deux bases multiples l'une de l'autre ont leur 
bases duales 
également multiples l'une de l'autre, le choix de $\vec{\mathscr{B}}$ 
parmi les repères adaptés à $\mathscr{R}$ est indifférent. On dit que 
$\mathscr{R}^*$ est le \emph{repère dual} de $\mathscr{R}$.
\end{defi}

On rappelle que $\Delta=(p_1,p_2,p_3)$ est un triangle de $\plan$
dont on note $D_1=(p_2p_3)$, $D_2=(p_3p_1)$, $D_3=(p_1p_2)$ les côtés.
Si $p$ est un point de $\plan$ qui est générique, c'est-à-dire qui
n'est sur aucun des côtés de $\Delta$, alors $\mathscr{R}_p=(p_1,p_2,p_3,p)$
forme un repère. Notons $\mathscr{R}_p^*=(q_1,q_2,q_3,q)$ son repère dual ;
on a alors $q_i=D_i^*$ pour $i=1,2,3$.

En effet, identifiant un point et ses coordonnées, on a $p_1=[1:0:0]$,
$p_2=[0:1:0]$, $p_3=[0:0:1]$ et $p=[1:1:1]$.
De même, dans les coordonnées homogènes de $\plan^*$ définies
par le repère dual, on a $q_1=[1:0:0]$,
$q_2=[0:1:0]$, $q_3=[0:0:1]$ et $q=[1:1:1]$. De plus, vu la relation
entre une base et sa base duale, un point $[a:b:c]$ de $\plan^*$
correspond à la droite de $V^*$ engendrée par la forme linéaire
qui s'écrit $(x_1,x_2,x_3)\mapsto ax_1+bx_2+cx_3$
dans les coordonnées définies par une base adaptée à $\mathscr{R}_p$.
Ainsi, $q_1=\project(\phi_1)$ où $\phi_1$ est
une forme linéaire non nulle s'annulant sur les droites
vectorielles $p_2$ et $p_3$, ce qui montre que $q_1=(p_2p_3)^*=D_1^*$.

De même $q=\project(\phi)$ où $\phi$ est la forme linéaire sur $V$
définie par $\phi(x_1,x_2,x_3)=x_1+x_2+x_3$ dans les coordonnées 
données par une base adaptée à $(p_1,p_2,p_3,p)$.
Il n'est pas à priori évident
d'exprimer ce point indépendamment des coordonnées (qui ici
dépendent de $p$).

\begin{defi}[polaire repère]
Si $p$ est un point de $\plan$ qui est générique,
on appelle \emph{droite polaire repère} de $p$ par rapport à $\Delta$
et on note $p^\rep$ la droite $q^*\subset\plan$ où
$(q_1,q_2,q_3,q)$ est le repère dual de $(p_1,p_2,p_3,p)$. 

Si $D$ est une droite ne contenant aucun sommet de $\Delta$,
on appelle \emph{point polaire repère} de $D$ par rapport à
$\Delta$ et on note $D^\rep$ le point $D^{*\rep *}$ où
la dualité repère dans $\plan^*$ est définie par rapport
au triangle de sommets $(D_{1}^*,D_{2}^*,D_{3}^*)$.
\end{defi}

\begin{prop}
La polarité repère est involutive~: pour tout point générique
$p^{\rep\rep}=p$.
\end{prop}

\begin{proof}
Le calcul simple précédent montre que $p^\rep$
a pour équation $x_1+x_2+x_3=0$ dans les coordonnées homogènes
$[x_1:x_2:x_3]$ définies par $\mathscr{R}_p$.

Notons $[x_1^*:x_2^*:x_3^*]$ les coordonnées
homogènes définies sur $\plan$ par le repère dual $\mathscr{R}_p^*$.
Par définition, ${p^\rep}^*=q$ donc $p^{\rep *\rep}$ est la droite
de $\plan^*$ d'équation $x_1^*+x_2^*+x_3^*=0$. Son point dual
a donc des coordonnées vérifiant $x_1^*x_1+x_2^*x_2+x_3^*x_3=0$
pour tous $x_i^*$ de somme nulle ; autrement dit
$p^{\rep *\rep *}=[1:1:1]=p$.
\end{proof}

La question suivante trouvera une réponse dans la prochaine section :
étant donnés $p$ et $\Delta$, peut-on construire $p^\rep$ à la règle ?

\section{Polarité harmonique}

La deuxième notion de polarité que nous abordons est traitée
par exemple dans le livre de Jean-Denis Eiden \cite{Eiden}.
Elle est spécialement intéressante ici car
facilement constructible à la règle.

\subsection{Le cas de la droite~: division harmonique}

Commençons par le cas très simple de la droite projective. Notons que
l'hyperplan dual d'un point est encore un point~; plutôt qu'un triangle
on considère donc une paire de points $p_1,p_2$, et on cherche une 
opération dépendant de cette paire qui envoie un point différent 
de $p_1$ et $p_2$ sur un autre tel point.
\begin{defi}
Supposons que $\project(V)$ est une droite projective.
On appelle \emph{quatrième harmonique} d'un repère
$\mathscr{R}_p=(p_1,p_2,p)$ le point $p'$ tel que le birapport 
$(p_1,p_2,p,p')$ vaille $-1$.
\end{defi}

Remarquons qu'en notant $q_1=p_2^*$, $q_2=p_1^*$ et $q=p'^*$,
on a $\mathscr{R}_p^*=(q_1,q_2,q)$. En effet, 
en se plaçant dans les coordonnées homogènes $[x_1:x_2]$
définies par une base adaptée à $\mathscr{R}_p$ 
on a $p_1=[1:0]$, $p_2=[0:1]$ et $p=[1:1]$. Dans la carte affine
$(x_2=1)$ ceci s'écrit encore $p_1=\infty$, $p_2=0$ et $p=1$.
Alors $p'$ est définit par
\[ \left.\frac{p'-0}{p'-\infty} \middle/ \frac{1-0}{1-\infty} \right.=-1\]
qui se réécrit $p'=-1$ par la convention $(a-\infty)/(b-\infty)=1$.
Ainsi en coordonnées homogènes $p'=[-1:1]$.
Alors dans le repère $\mathscr{R}_p^*$, les coordonnées homogènes
de $q=p'^*$ vérifient $-x^*+y^*=0$, et $q=[1:1]$ comme souhaité.

\subsection{Le cas du plan}

Dans le cas du plan, par intersection on obtient des triplets
de points sur chacun des côtés du triangle $\Delta$. Imitant
le cas de la dimension $1$, on peut alors construire trois nouveaux
points, qui engendrent une droite qu'on va donc définir
comme la polaire harmonique. 

\begin{defi}
Soit $\Delta=(p_1,p_2,p_3)$ un triangle d'un plan projectif $\plan$ et
$p$ un point générique.
Pour toute permutation $(i,j,k)$ de $(1,2,3)$, notons $u_i$ le
point d'intersection de $(p_i p)$ avec $(p_j p_k)$ et $u_i'$ le quatrième
harmonique du triplet $(p_j,p_k,u_i)$ (voir figure \ref{fig:harmonique1}).
Alors les trois points $u'_1$, $u'_2$ et $u'_3$ sont alignés, et la droite
qu'ils définissent est appelée la \emph{polaire harmonique}
de $p$ par rapport au triangle $\Delta$ et est notée $p^\#$.
\end{defi}

\begin{figure}\begin{center}
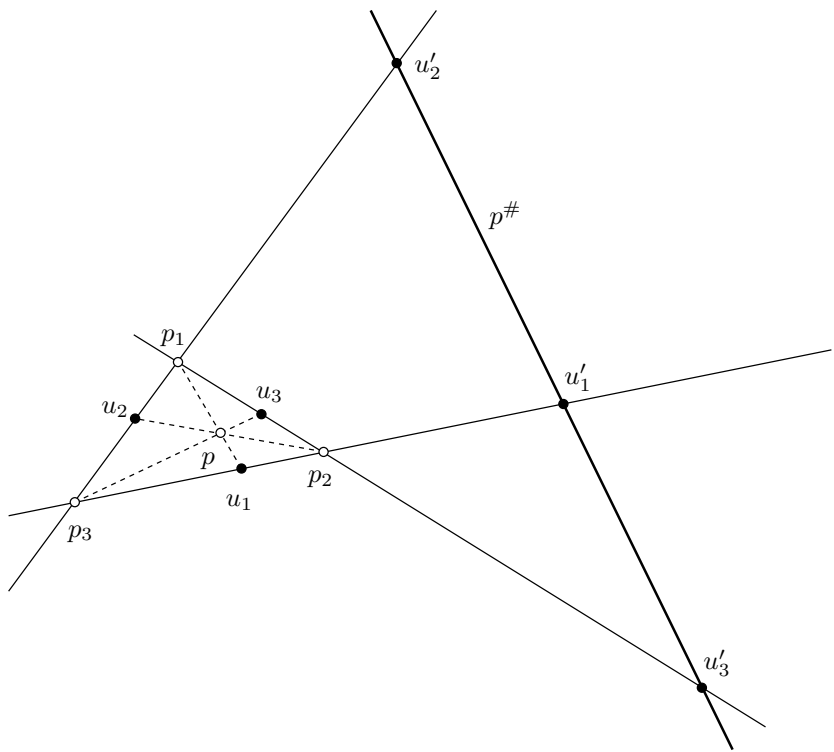
\caption{Définition de la polaire harmonique $p^\#$ d'un point $p$.}
\label{fig:harmonique1}
\end{center}\end{figure}

Démontrons l'alignement annoncé. 
On se place dans les coordonnées $[x_1:x_2:x_3]$ et $[x_1^*:x_2^*:x_3^*]$
définies par les repères $\mathscr{R}_p=(p_1,p_2,p_3,p)$ et
$\mathscr{R}_p^*$.
Comme précédemment, pour toute permutation $(i,j,k)$ de $(1,2,3)$
$p_i$ a pour coordonnées $x_j=x_k=0$ et $x_i=1$,
et $p=[1:1:1]$. La droite d'équation $(x_i=0)$ contient
$p_j$ et $p_k$, c'est donc la droite $D_i$. De même,
$(p_ip)$ a pour équation $x_j=x_k$, de sorte que leur intersection
$u_i$ a pour coordonnées $x_i=0,x_j=x_k$ (par exemple on peut prendre
$x_j=x_k=1$). On en déduit 
que les coordonnées de $u'_i$ sont $x_i=0,x_j=-1,x_k=1$.

Les trois points $u'_i$ obtenus sont sur la droite $D$ d'équation 
$x_1+x_2+x_3=0$ donc sont bien alignés.
De plus, on remarque que cette droite est égale à $p^\rep$.

\begin{prop}\label{prop:harmnoique}
Pour tout triangle $\Delta$ et tout point générique $p$,
on a $p^\rep=p^\#$.
\end{prop}

On définit le polaire harmonique d'une droite générique $D$
par $D^\#=D^{*\# *}$, de sorte que par la proposition \ref{prop:harmnoique}
et l'involutivité de $\null^\rep$, le point $D^\#$ se construit de la
façon suivante. On considère les
points d'intersections $u'_i$ de $D$ avec $D_i$, et les
quatrièmes harmoniques $u_i$ de $(p_j,p_k,u'_i)$. Alors
les droites $(p_i u_i)$ sont concourantes en $D^\#$.

Par ailleurs, on sait que dans un plan on peut 
construire à la règle le quatrième harmonique de trois
points alignés (voir la figure \ref{fig:construction}).
Comme $(p_j,p_k,u_i,u'_i)$ sont en division harmonique si
et seulement si $(p_j,p_k,u'_i,u_i)$ le sont, on obtient
une autre construction du quatrième harmonique~: partant 
de $(p_j,p_k,u'_i)$, on construit les droites de la figure dans l'ordre
4,1,3,2 en ayant choisi des points $m$ et $n\in (u'_im)$
pour obtenir $u_i$.

\begin{figure}[htp]\begin{center}
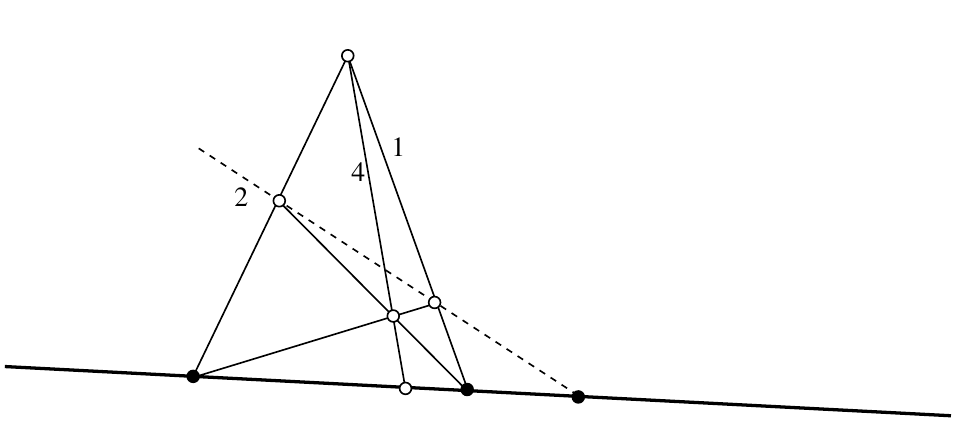
\caption{Construction à la règle du quatrième harmonique $u'_i$ de 
$(p_j,p_k,u_i)$. Les numéros des droites indiquent l'ordre
de constuction. Le point $m$ et la droite en pointillés sont
arbitraires, cette dernière devant toutefois passer par
$u_i$.}\label{fig:construction}
\end{center}\end{figure}

Dans notre cas, plutôt que de prendre $m$ et $n$ arbitraires on peut
utiliser $p_i$ et $p$, et on constate que $u'_i$ est l'intersection
de $(p_jp_k)$ avec $(u_ju_k)$. L'alignement démontré en coordonnées
plus haut n'est donc rien d'autre que le théorème de Desargues
pour les triangles $(p_1,p_2,p_3)$ et $(u_1,u_2,u_3)$ et le point
$p$. La construction qu'on a donnée a l'avantage de mettre en
évidence le lien avec la dimension $1$.

Une dernière remarque amusante~: si $\mathbb{K}$ est de caractéristique
$2$, on a $-1=1$ donc $u'_i=u_i$, et les points $(u_1,u_2,u_3)$ sont
alignés. On reconnait dans la configuration $(p_1,p_2,p_3,u_1,u_2,u_3,p)$
le plan de Fano~; dans le cas $\mathbb{K}=\mathbb{F}_2$, ces $7$
points sont tous les points de $\plan$.

\section{Polarités cubiques}

Dans cette section on s'intéresse à une polarité plus algébrique, qui généralise
celle, classique, de la polarité par rapport à une conique. Pour des 
raisons qui s'éclairciront, on suppose dans toute cette partie que le corps 
$\bK$ est de caractéristique différente de $2$ et $3$. On se place sur un plan 
projectif dans des coordonnées homogènes $[x_1:x_2:x_3]$.

\subsection{Polarité par rapport à une conique}

Par {\it conique} on entend ici la donnée d'une forme quadratique $Q$
sur l'espace vectoriel $V$ à multiplication par un scalaire non nul près.
Parfois, on considère que la conique est le projectivé du cône isotrope de 
$Q$, et dans beaucoup de cas ce cône détermine $Q$ à un facteur près.

La forme $Q$ admet (en caractéristique différente de $2$)
une forme polaire $\varphi$, définie comme l'unique forme bilinéaire 
symétrique telle que $\varphi(\vec u,\vec u)=Q(\vec u)$ pour tout $\vec u\in V$. 
Ceci découle de l'identité de polarisation
\[2\varphi(\vec u,\vec v)= Q(\vec u+\vec v)-Q(\vec u)-Q(\vec v).\]
Multiplier $Q$ par un scalaire amène le même changement pour $\phi$,
de sorte que la donnée d'une conique détermine la forme polaire à constante
près.

La contraction qui à $\vec u$ associe la forme linéaire $\varphi(\vec u,\cdot)$ 
est une application linéaire de $V$ dans $V^*$. On définit la 
{\it polaire} (par rapport à la conique donnée) d'un point $p=\project(\vec u)$
qui n'est pas dans le noyau de $\varphi$ comme le projectivé 
de l'orthogonal pour $Q$ de $\vec u$, autrement dit l'ensemble
\[p^\perp := \{\project(\vec v)\,;\,\varphi(\vec u,\vec v)=0\}.\] 
Comme $\vec u$ n'est pas dans le noyau de $\varphi$, cet ensemble est le 
projectivé d'un plan de $V$, donc une droite projective. Dans le cas où
$Q$ est non-dégénérée, le noyau de $\varphi$ est réduit à $\vec 0$ donc
la droite polaire de tous les points de $\project(V)$ est définie. Dans ce cas,
l'application $\vec u\mapsto \varphi(\vec u,\cdot)$ est un isomorphisme de
$V$ vers $V^*$ donc l'application $p\mapsto p^\perp$ est un isomorphisme
de $\project(V)$ vers $\project(V^*)$.

De même, le point polaire d'une droite $D=\project(W)$
est le point correspondant à l'orthogonal pour $Q$ de $W$.
Ceci revient à composer la dualité et l'identification entre $\project(V)$
et $\project(V^*)$ donnée par la contraction.

\subsection{Le degré $3$}

Il est injustement peu connu que tout le paragraphe précédent se généralise en
degré supérieur. Une {\it cubique} est la donnée d'une forme homogène $C$ de 
degré $3$ (ou forme cubique) sur $V$
à multiplication par un scalaire non nul près. On confond parfois une forme
cubique avec l'ensemble des points $p=\project(\vec u)$ pour lesquels $C(\vec u)=0$.

Comme une forme quadratique, $C$ admet une {\it forme polaire}~:
l'unique forme trilinéaire symétrique
$\psi:V\times V\times V\to \bK$ telle que 
$\psi(\vec u,\vec u,\vec u)=C(\vec u)$ pour tout $\vec u\in V$.
En effet, un calcul direct montre qu'une telle forme doit nécessairement
satisfaire l'identité de polarisation
\begin{align*}
6\psi(\vec u,\vec v,\vec w)= &\ C(\vec u+\vec v+\vec w)-C(\vec u+\vec v)-
  C(\vec u+\vec w)\\
  &-C(\vec v+\vec w)+ C(\vec u)+C(\vec v)+C(\vec w),
\end{align*}
ce qui montre l'unicité. L'existence se montre facilement
dans le cas où $C$ est un monôme~: 
\[(\vec u,\vec v, \vec w)\mapsto u_i v_i w_i\]
est la forme polaire de $\vec u\mapsto u_i^3$~; 
\[(\vec u,\vec v, \vec w)\mapsto \frac13 (u_iv_iw_j+ u_iv_jw_i+u_jv_iw_i)\]
est celle de $\vec u\mapsto u_i^2u_j$, etc. Par combinaison linéaire,
on en déduit l'existence de la forme polaire dans le cas général.

On dispose maintenant de deux contractions~:
l'application linéaire $\vec u\mapsto \psi(\vec u,\cdot,\cdot)$
de $V$ dans l'espace $\mathrm{S}^2(V,\bK)$ des formes bilinéaires symétriques 
sur $V$ et la forme bilinéaire symétrique
$(\vec u,\vec v)\mapsto \psi(\vec u,\vec v,\cdot)$ à valeur dans $V^*$. On 
appelle ces applications la première et la deuxième contraction de $\psi$.

On définit trois niveau d'annulation de $\psi$. Son \emph{premier noyau}
(ou simplement \emph{noyau}) est le noyau de sa première contraction,
autrement dit l'ensemble
\[\ker\psi := \{\vec u\in V\,;\, \psi(\vec u,\vec v,\vec w)=0\quad
  \forall\vec v,\vec w\in V\}~;\]
c'est un sous-espace vectoriel de $V$. 
Son \emph{second noyau} est l'ensemble des vecteurs qui, contractés deux
fois avec $\psi$, l'annulent~:
\[\ker^2\psi := \{\vec u\in V\,;\, \psi(\vec u,\vec u, \vec v)=0\quad
  \forall \vec v\in V\}~;\]
c'est l'intersection de trois coniques (éventuellement dégénérées) et
il contient le premier noyau.
Enfin, son \emph{troisième noyau} (ou encore \emph{cone isotrope}) est l'ensemble
des vecteurs qui, contractés trois fois avec $\psi$, l'annulent~:
\[\ker^3\psi := \{\vec u\in V\,;\, \psi(\vec u,\vec u,\vec u)=0\}~;\]
le cône isotrope contient donc le second noyau et son projectivé est
l'ensemble des points de la cubique donnée. 

Si $p=\project(\vec u)$ est tel que $\vec u\notin \ker\psi$, on définit
sa \emph{première polaire} comme la conique définie par
$Q(\vec v)=\psi(\vec u,\vec v,\vec v)$ et on la note $p'$.
Si de plus $\vec u\notin\ker^2\psi$, on définit la \emph{seconde polaire} de
$p$ comme la droite définie par la forme linéaire $\psi(\vec u,\vec u,\cdot)$
et on la note $p''$. Notons que $p''$ est également la droite polaire
de $p$ par rapport à la conique $p'$.

L'application qui à un point associe sa première polaire
est loin d'être surjective, puisque dans le meilleur des cas (quand le noyau
de $\psi$ est vide) c'est un plongement projectif du plan $\plan$ dans l'espace
des coniques de $V$, qui est de dimension $5$. On ne définit donc pas de
réciproque à la première polarité. L'application de seconde polarité n'est en
général pas surjective non plus, mais on pourrait vu les dimensions en jeu 
définir une réciproque sur une partie de $\project(V^*)$
(identifié à l'ensemble des droites de $\project(V)$). On ne le fera ici
que dans un cas particulier.

\subsection{Le cas d'un triangle}

Venons-en au lien qu'entretien la polarité algébrique avec 
la dualité des repères. On considère toujours un triangle
$\Delta=(p_1,p_2,p_3)$, un point générique $p$ et une droite générique
$D$.

\begin{prop}
Soit $C$ la cubique réunion des trois droites $D_3=(p_1p_2)$, 
$D_2=(p_1p_3)$ et $D_1=(p_2p_3)$.
Son premier noyau est réduit à $0$ et son second noyau a 
pour projectivé l'union des trois sommets $\{p_1,p_2,p_3\}$,
de sorte que les deux polaires de
$p$ par rapport à $C$ sont bien définies.

De plus la seconde polarité est bijective depuis l'ensemble
des points distincts des sommets, vers les droites distinctes
des côtés.

On note $p^\perp$ la seconde polaire de $p$ par rapport à
$C$, et $D^\perp$ le point dont la seconde polaire est
$D$, et on les appelles les polaires cubiques de $p$ et $D$
par rapport à $\Delta$. On a alors $p^\perp=p^\rep$
et $D^\perp=D^\rep$.
\end{prop}

En particulier, on voit que la polarité $p\mapsto p^\rep$ s'étend
aux côtés de $\Delta$, mais pas à ses sommets. Remarquons qu'il
découle de cette proposition que $D^\perp=D^{*\perp *}$.

\begin{proof}
On se place dans les coordonnées homogènes $[x_1:x_2:x_3]$ définies
par $\mathscr{R}_p$.

Comme la droite $D_k=(p_ip_j)$ a pour équation $x_k=0$,
la cubique $C$ a pour équation $x_1x_2x_3=0$ et pour forme polaire,
à un facteur
$6$ près,
\[\psi\left(\left|\begin{array}{c} x_1\\x_2\\x_3\end{array}\right.,
  \left|\begin{array}{c} y_1\\y_2\\y_3\end{array}\right.,
  \left|\begin{array}{c} z_1\\z_2\\z_3\end{array}\right. \right)
  = \sum_{(i,j,k)} x_i y_j z_k \]
où $(i,j,k)$ parcourt les permutations de $(1,2,3)$.
Autrement dit, si on identifie $(\mathbb{K}^3)^3$ à l'ensemble
des matrices $3\times 3$, $\psi$ est (à un facteur près) le \emph{permanent}.

Le premier noyau s'en déduit immédiatement, et on voit que le second
a pour système d'équations
\[\left\{ \begin{matrix} x_1x_2 = 0 \\ x_1x_3 = 0 \\ x_2x_3 =0
\end{matrix}
\right.\]
donc est réduit aux points ayant deux coordonnées nulles, c'est-à-dire
aux sommets de $\Delta$.

En entrant deux fois les coordonnées $[1:1:1]$ de $p$
dans l'expression de $\psi$ on voit que sa seconde polaire 
par rapport à $C$ a pour équation
$x_1+x_2+x_3 = 0$. Ainsi, $p^\perp=p^\rep$.
Le cas de $D$ s'en déduit par dualité. 
\end{proof}

\subsection{Détour par la transformation de Crémona}

La première polaire de $p$ a pour équation $x_1x_2+x_1x_3+x_2x_3=0$,
donc passe par $p_1$, $p_2$ et $p_3$ ; de plus elle est tangente
en $p_i$ à la droite d'équation $x_j+x_k=0$, qui
passe par le quatrième harmonique sur $(p_jp_k)$ des points
$p_j$, $p_k$ et $(p_jp_k)\cap(p_ip)$.
Ainsi, si $\bK=\bR$,
dans une carte affine telle que les points $p_1,p_2,p_3$ forment un triangle 
équilatéral
dont $p$ est le centre de gravité, cette conique est
le cercle circonscrit au triangle (voir figure \ref{fig:construction2}).

\begin{figure}[htp]\begin{center}
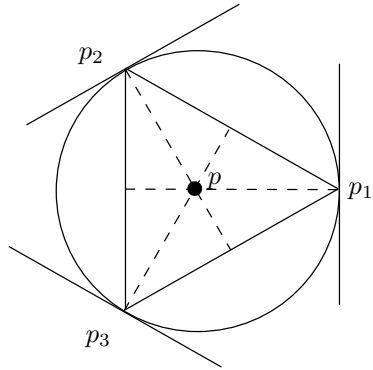
\caption{Conique polaire de $p$ par rapport au triangle $(p_1,p_2,p_3)$~: les
         quatrièmes harmoniques sont à l'infini.}\label{fig:construction2}
\end{center}\end{figure}

On appelle alors la première polaire de $p$ l'\emph{ellipse circonscrite}
à $\Delta$ de {\it centre} $p$. Fixons $p$ et notons $E$ cette conique.
Si l'on considère l'application $\Phi$ composée de la dernière polarité par 
rapport à $\Delta$ et de la polarité par rapport à $E$ (cette dernière étant 
simplement destinée à manipuler des points plutôt que des droites),
on peut vérifier qu'elle s'écrit dans les coordonnées définies par
$\mathscr{R}_p$ :
\[\Phi([x_1:x_2:x_3])=[x_2x_3:x_1x_3:x_1x_2].\]
C'est donc une transformation rationnelle (c'est-à-dire, à coefficients
polynomiaux homogènes de même degré), et elle est bien définie
hors des sommets de $\Delta$~; elle est appelée transformation de Crémona
et c'est une involution (elle est en particulier d'inverse rationnel).
L'application $\Phi$ décrit comment dégénère la polarité par rapport
à $\Delta$ au bord de son domaine de définition ; par exemple elle envoie
$D_i$ sur $p_i$, et si $q_n\to p_i$ avec une direction asymptotique
$(p_i q_n)\to D$, alors $\Phi(q_n)$ tends vers un point de $D_i$.

\section{Polarité convexe}

Dans cette section, $\mathbb{K}=\mathbb{R}$. Un \emph{convexe} de $\plan$
est une partie dont l'intersection avec toute droite projective est
connexe. Un convexe $C$ est \emph{propre} s'il existe une droite qui ne rencontre
pas son adhérence; autrement dit, si $C$ est relativement compact dans une carte
affine. La polarité convexe est facile à définir pour les droites.

\begin{defi}
Soit $C$ un convexe ouvert propre de $\plan$ et $D$ une droite projective
ne rencontrant pas $\bar C$. On appelle point polaire de $D$
par rapport à $C$ et on note $D^\circ$ le barycentre de $C$ dans 
n'importe quelle carte affine envoyant $D$ à l'infini.
\end{defi}

Cette définition est cohérente car la notion de barycentre est affinement
équivariante et que deux cartes affines ayant la même droite à l'infini
diffèrent par une transformation affine. 

La définition de la droite polaire d'un point est alors obtenue par dualité.
\begin{defi}
Le convexe dual $C^*$ de $C$ est défini
comme l'ensemble des droites de $\plan$ qui ne rencontrent pas
$\bar C$. C'est un convexe ouvert propre de $\plan^*$.
Étant donné un point $x$ de $C$, la droite $x^*$ de $\plan^*$ ne rencontre pas l'adhérence
de $C^*$, donc on peut définir son point polaire convexe $x^{*\circ}$
(par rapport à $C^*$). On appelle alors polaire convexe
de $x$ par rapport à $C$ la droite $x^\circ:=x^{*\circ *}$ de $\plan$~; elle
ne rencontre pas l'adhérence de $\plan$.
\end{defi}

Cette définition, très naturelle, n'est toutefois pas involutive :
on n'a en général pas $x=x^{\circ\circ}$ !
Yves Benoist a en effet remarqué ceci dans le cas d'un quadrilatère \cite{Benoist}.

Avant de traiter le cas d'un triangle, nous allons déterminer par des outils
affines l'inverse de l'application $D\to D^\circ$.

\subsection{Convexes affines, dualité et point de Santaló}

On se place ici en dimension quelconque car les notations et idées ne
s'en trouvent pas alourdies pour autant. L'essentiel de ce qui suit
peut être consulté et complété dans \cite{Faraut-Koranyi}.

Considérons un ouvert convexe relativement compact $K$ de l'espace affine
$\mathbb{R}^n$. On définit pour chaque point $x\in K$ le dual
de $K$ centré en $x$ de la façon suivante :
\[K^x = \left\{f\in(\mathbb{R}^n)^* \,\middle|\, \forall y\in \bar K\quad 
f(\vec{xy})>-1 \right\}.\]
C'est un ouvert convexe relativement compact du dual de $\mathbb{R}^n$ qui
contient l'origine $o$ dans son intérieur. Notons qu'il est courant de considérer
plutôt $-K^x$, mais la définition ci-dessus simplifie légèrement ce qui suit.

Le \emph{point de Santaló} de $K$ est défini comme l'unique point $x\in K$
minimisant le volume de $K^x$. Voyons comment l'introduction d'un cône convexe
et de sa \emph{fonction caractéristique} permet à la fois de montrer l'existence
et l'unicité du point de Santaló et de construire l'inverse de la polarité convexe.

\subsubsection{Cônes convexes}

\begin{figure}\begin{center}
\input{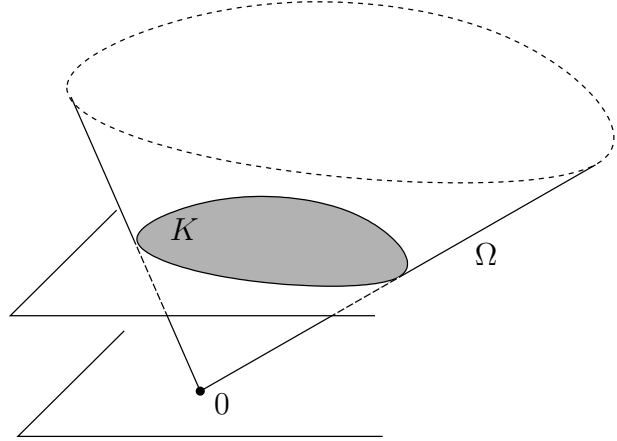}
\caption{Le cône vectoriel $\Omega$ associé à un convexe $K$.}
\label{fig:cone}
\end{center}\end{figure}

On identifie $\mathbb{R}^n$ avec l'hyperplan affine $H$ d'équation $(t=1)$
de $\mathbb{R}^{n+1}$ dont les coordonnées canoniques sont notées $(x_1,\ldots,x_n,t)$,
et on considère le cône convexe s'appuyant sur $K$ (voir figure \ref{fig:cone}) :
\[\Omega := \left\{(x,t)\in\mathbb{R}^{n+1} \,\middle|\, t>0 \mbox{ et } x/t\in K\right\}.\]
Le \emph{dual} de $\Omega$ est le cône convexe de $(\mathbb{R}^{n+1})^*$ défini par 
\[\Omega^* := \left\{f\in(\mathbb{R}^{n+1})^* \,\middle|\, 
  f(y)>0\quad\forall y\in \bar\Omega\setminus\{0\}\right\}\]
de sorte que $\project(\Omega^*)\subset\project((\mathbb{R}^{n+1})^*)$ est le convexe dual de
$\project(\Omega)\subset\project(\mathbb{R}^{n+1})$.

La donnée de $\Omega$ et $H$ donne par dualité $\Omega^*$ et un point
$\tau=\tau^H\in(\mathbb{R}^{n+1})^*$ défini par $\tau(x,t)=t$, mais ne fournit
pas d'hyperplan affine naturel ; c'est là que le choix d'un point dans
$K$ intervient.

En effet, étant donné $x\in K$ on peut définir un hyperplan affine par
\[H^x:=\left\{f\in(\mathbb{R}^{n+1})^* \,\middle|\, f(x,1)=1\right\}.\]
Ainsi, la donnée de $(\Omega,H,x)$ donne par dualité une donnée de même
nature $(\Omega^*,H^x,\tau^H)$ dans le dual. De plus, si on note
$\pi=\pi_H^x$ la projection linéaire sur la direction $\vec H$
de $H$ dans la direction de $\vec{ox}$, l'application
\[f\in K^x \mapsto \tilde f=f\circ\pi+\tau\]
identifie $K^x$ à $\Omega^*\cap H^x$. Ceci permet de lever la difficulté
liée à la dépendance de $K^x$ en $x$.

\subsubsection{Fonction caractéristique}

La fonction caractéristique d'un cône convexe est la fonction
$\varphi:\Omega\to \mathbb{R}$ définie par
\[\varphi(x):=\int_{\Omega^*} e^{-f(x)} df\]
où la mesure est la mesure de Lebesgue ; $\mathbb{R}^{n+1}$ et son dual
sont ici munis de leurs produits scalaires canoniques, mais ceci n'a pas
grande importance puisque la multiplication de $\varphi$ par une constante
ne change pas ce qui suit.

Tout d'abord, on montre que $\varphi$ est strictement log-convexe.
En effet, si $u$ est un vecteur non nul de $\mathbb{R}^{n+1}$, on obtient
pour $D^2_{uu}\log\varphi(x)$ l'expression
\[\frac1{\varphi(x)^2}
  \left(\int_{\Omega^*} e^{-f(x)}df\int_{\Omega^*} f(u)^2e^{-f(x)} df
  -\left(\int_{\Omega^*}f(u)e^{-f(x)} df\right)^2\right)\]
et il suffit d'écrire le dernier intégrande sous la forme
$e^{-f(x)/2}\, f(u)e^{-f(x)/2}$ pour que l'inégalité de Cauchy-Schwarz assure
$D^2_{uu}\log\varphi(x)>0$.

La fonction caractéristique permet aussi de définir une application
de $\Omega$ dans $\Omega^*$ par $\theta_\Omega(x):= -d\log\varphi(x)$.
En effet, en dérivant sous l'intégrale on voit que
\[\theta_\Omega(x)=\frac{\int_{\Omega^*} f e^{-f(x)}df}{\int_{\Omega^*}e^{-f(x)}df}\]
de sorte que $\theta_\Omega(x)(y)>0$ pour tout $y\in\bar\Omega\setminus\{0\}$.

Montrons maintenant que $\theta_\Omega$ est bijective. Elle est injective
par stricte log-convexité de $\varphi$. Soit $g\in\Omega^*$ et considérons
$H^{*g}$ l'hyperplan de $\mathbb{R}^{n+1}$ d'équation $g(x)=1$.
Par convexité, et comme $\varphi$ tend vers l'infini
au bord de $\Omega$,\footnote{Voir par exemple la proposition I.3.2 de \cite{Faraut-Koranyi}}
il existe un unique point $y$ de $\Omega\cap H^{*g}$ qui y 
minimize $\varphi$. Alors $d\log\varphi(y)$ a pour noyau la direction
de $H^{*g}$, de sorte que $\theta_\Omega(y)$ est un multiple
de $g$. Or par changement de variable, $\varphi$ est homogène de degré $-(n+1)$ donc
on a $[d\varphi(y)](y)=-(n+1)\varphi(y)$, et $\theta_\Omega(y)=(n+1)g$.
Mais $\theta_\Omega$ est homogène de degré $-1$, donc $\theta_\Omega((n+1)y)=g$. 

\subsubsection{Conclusion}

Relions maintenant la fonction caractéristique et le volume de $K^x$.
Il suffit pour cela d'écrire $\Omega^*$ comme la réunion disjointe des
$t K^x=\Omega^*\cap tH^x$ et d'appliquer le théroème de Fubini :
\begin{eqnarray}
\varphi(x)&=&\int_0^{+\infty} \int_{tK^x} e^{-f(x)} df \,dt\\
  &=&\int_0^{+\infty} \int_{K^x} e^{-tg(x)} t^n dg \,dt\\
  &=&\mathrm{vol}(K^x)\int_0^{+\infty} e^{-t} t^n dt.
\end{eqnarray}
Comme 
$\int_0^{+\infty} t^{n+1}e^{-t}dt=(n+1)\int_0^{+\infty} t^ne^{-t} \,dt$,
on trouve par le même argument
\[\theta_\Omega(x)=\frac{n+1}{\mathrm{vol}(K^x)}\int_{K^x} g \,dg.\]
Ainsi, $\theta_\Omega(x)$ est toujours le centre de gravité de
$(n+1) K^x$.

Étant donné $g\in \Omega^*$, on peut donc trouver un hyperplan affine
$H^*$ de $(\mathbb{R}^{n+1})^*$ tel que $g$ soit le centre de gravité
de $\Omega^*\cap H^*$. En effet, il suffit de considérer le convexe
$K'=\Omega\cap H^{*g}$ et de considérer son point de Santaló $x$ :
alors $x$ minimise $\varphi$ sur $K'$ donc $\theta_\Omega(x)$ est
un multiple de $g$ ; l'hyperplan parallèle à $H^x$ qui contient
$g$ convient.

Le passage des cônes convexes aux convexes projectifs étant immédiat
on a montré le résultat suivant, qui m'a été suggéré par Greg Kuperberg
et est sans doute plus ou moins connu.
\begin{theo}\label{theo:convexe}
Soit $K$ un ouvert convexe propre d'un espace projectif réel $\project(V)$
et $x$ un point de $K$. Soit $K^*\subset \project(V^*)$ le convexe dual
de $K$, et $y$ son point de Santaló quand il est vu dans une carte affine
envoyant $x^*$ à l'infini. Alors l'hyperplan $x^\circ:=y^*$ de
$\project(V)$ a la propriété suivante : dans toute carte affine
envoyant $x^\circ$ à l'infini, $x$ est le barycentre de $K$.
\end{theo}

\subsection{Deux questions}

Ce qui précède amène deux questions qui semblent intéressantes, et dont je
n'ai pas trouvé de trace dans la littérature.

La première est de nature dynamique.
\begin{ques}
Soit $C$ un ouvert convexe projectif propre (en dimension quelconque). Que
peut-on dire de la dynamique de l'application
\begin{eqnarray*}
\circ^2 \colon  C &\to& C \\
 x &\mapsto& x^{*\circ * \circ}
\end{eqnarray*}
où les polarités sont entendues par rapport à $C$ et $C^*$~? 
\end{ques}
Ainsi, on peut se demander pour quels $C$ elle est d'ordre fini (et en particulier
dynamiquement triviale), chercher à décrire ses points fixes ou périodiques, etc.

Ensuite, on a vu que le centre de gravité pouvait être remplacé par le point de
Santaló et donner naissance à une polarité au même titre. Plus généralement,
appelons «~point spécial affine~» (de dimension $n$) toute application
$\gamma$ de l'ensemble des ouverts convexes bornés de $\mathbb{R}^n$ vers $\mathbb{R}^n$
telle que~:
\begin{enumerate}
\item pour tout convexe $C$, $\gamma(C)\in C$ ;
\item $\gamma$ est affinement équivariante~: $\gamma(A(C))=A(\gamma(C))$ pour
  toute transformation affine $A$ et tout convexe $C$.
\end{enumerate}
Un point spécial affine $\gamma$ et un ouvert convexe propre $C$
donnent naissance à une polarité projective convexe~:
à une droite $D$ ne rencontrant pas $\bar C$, on associe $D^\gamma:=\gamma(C)$
où $C$ est lu dans n'importe quelle carte affine envoyant $D$ à l'infini~;
à un point $x\in C$ on associe la droite $x^\gamma:=x^{*\gamma *}$.
Le théorème \ref{theo:convexe} dit exactement qu'en notant $\null^\square$
la polarité définie par le point de Santaló (qui est un point spécial affine),
quelque soit le convexe $C$ on a $x^{\square\circ}=x$. La question suivante
m'a été suggérée par Anne Parreau.
\begin{ques}
Existe-t-il un point spécial affine $\gamma$ définissant une polarité
involutive, c'est-à-dire tel que quelque soit l'ouvert projectif
convexe propre $C$ et pout tout $x\in C$, on ait $x^{\gamma\gamma}=x$~?
Si oui, peut-on caractériser l'ensemble des tels $\gamma$~?
\end{ques}

\subsection{Le cas d'un triangle}

La définition de la polarité convexe par rapport à un triangle est maintenant
facile à imaginer.
\begin{defi}
La droite polaire convexe $p^\circ$ d'un point générique $p\in\plan$ par rapport
à un triangle $\Delta=\{p_1,p_2,p_3\}$ est la polaire convexe de $p$
par rapport à la composante connexe de $\plan\setminus(D_1\cup D_2\cup D_3)$
qui contient $p$. 

Le point polaire $D^\circ$ d'une droite générique par rapport à $\Delta$
est le polaire de $D$ par rapport à l'unique composante connexe de
$\plan\setminus(D_1\cup D_2\cup D_3)$ qui ne rencontre pas $D$.
\end{defi}
Notons que pour voir qu'il y a bien une unique composante ne rencontrant pas une
droite générique donnée, il suffit de se placer dans une carte affine envoyant
cette droite à l'infini : $\plan\setminus(D_1\cup D_2\cup D_3)$
a alors une seule composante bornée.

Il est également important de remarquer qu'ici, tous les points spéciaux coïncident.
En effet, un triangle affine admet une transformation affine qui le laisse invariant,
mais ne fixe que son centre de gravité. Tout point spécial doit être invariant
par une telle transformation, donc est égal au centre de gravité. En particulier,
le point de Santaló d'un triangle est égal à son centre de gravité et
la discussion précédente montre que pour tous $p$ et $D$ génériques,
$p^{\circ\circ}=p$ et $D^{\circ\circ}=D$. Cette propriété d'involutivité est
vraie pour tout convexe homogène, c'est-à-dire tel que le groupe des transformations
projectives le conservant agisse transitivement sur son intérieur.

\begin{prop}
On a $p^\circ= p^\rep$ pour tout point générique $p$.
\end{prop}
Par involutivité, le même résultat est vrai pour les droites.

\begin{proof}
On utilise les coordonnées homogènes $[x_1:x_2:x_3]$ définies par
le repère $(p_1,p_2,p_3,p)$ et on considère la carte affine 
\[(y,z)\mapsto [1-y-z:y:z]\]
qui envoie à l'infini $p^\rep$, d'équation $(x_1+x_2+x_3=0)$,
et dont l'inverse hors de cette droite s'écrit
\[[x_1:x_2:x_3]\mapsto \left(\frac{x_2}{x_1+x_2+x_3},\frac{x_3}{x_1+x_2+x_3}\right).\]

Dans cette carte affine, et toujours en identifiant un point
et ses coordonnées, on a $p_1=(0,0)$, $p_2=(1,0)$, $p_3=(0,1)$ et
$p=(\frac13,\frac13)$, de sorte que $p$ est le centre de gravité
du triangle de sommets $p_1,p_2,p_3$. La droite à l'infini
est donc la polaire convexe de $p$, et $p^\circ=p^\rep$.
\end{proof}

\section{Généralisation en toute dimension}

Considérons maintenant un espace projectif $\project(V)$ de dimension
$n\geqslant 2$, un \emph{simplexe} $\Delta=(p_1,\dots,p_{n+1})$
dont les \emph{faces} sont les hyperplans $H_i$ engendrés par 
$\Delta_i:=(p_1,\dots,\hat{p}_i,\dots,p_{n+1})$ 
où l'élément avec un chapeau est 
absent. Un point est générique s'il n'appartient à aucune face
et un hyperplan est générique s'il ne passe par aucun sommet.

\subsection{Polarité repère}

On procède exactement comme en dimension $2$ : un point générique
$p$ définit un repère projectif $\mathscr{R}_p=(p_1,\dots,p_{n+1},p)$,
qui définit un repère dual $\mathscr{R}_p^*=(q_1,\dots,q_{n+1},q)$
de $\project(V^*)$ avec $q_i=H_i^*$. Alors le polaire de $p$
par rapport à $\Delta$ est l'hyperplan $p^\rep=q^*$ de $\project(V)$.
Si $H$ est un hyperplan générique on définit $H^\rep=p^{*\rep*}$
où la polarité repère dans le dual est par rapport au simplexe
$\Delta^*:=(H_1^*,\dots,H_{n+1}^*)$. On a encore
$p^{\rep\rep}=p$.

\subsection{Polarité harmonique}

Il y a plusieurs façons équivalentes de définir la polarité harmonique.
On peut par exemple procéder par récurrence : comme
$\Delta_i$ est un simplexe de $H_i$, on considère
les points $u_i=H_i\cap(p_ip)$, et leurs hyperplans polaires
$U_i\subset H_i$ par rapport à $\Delta_i$. Alors les $U_i$,
qui sont des sous-espaces projectifs de dimension $n-2$ de $\project(V)$,
sont sur un même hyperplan qu'on note $p^\#$ et qu'on appelle
son polaire harmonique par rapport à $\Delta$. La construction de
$H^\#$ est analogue.

On peut aussi considérer les points 
\[v_{ij}=(p_ip_j)\cap(p_1 \cdots \hat{p}_i \cdots \hat{p}_j \cdots p_{n+1} p)\]
et les quatrièmes
harmoniques $v'_{ij}$ des triplets $(p_i,p_j,v_{ij})$. Alors
les $v'_{ij}$ sont sur un même hyperplan, qui se trouve être
$p^\#$.

Comme précédemment, on a $p^\#=p^\rep$.

\subsection{Polarités algébriques}

Cette notion de polarité se relie à celle de {\it dernière polarité} par 
rapport
à une hypersurface de degré $n+1$. Une telle hypersurface $\Sigma$  possède
une équation de degré $n+1$ qui a comme forme polaire (définie si la caractéristique
de $\mathbb{K}$ est nulle ou strictement supérieure à $(n+1)$) $\psi$ une forme $(n+1)$-
linéaire symétrique.
Étant donné un point $p$, projeté d'un vecteur $\vec u$, on peut contracter
$n$ fois $\psi$ avec $\vec u$, ce qui donne une forme linéaire
$\psi(\vec u,\vec u,\ldots,\vec u,\cdot)$ bien définie à un facteur près.
Si cette forme est non nulle, son noyau définit
un hyperplan qu'on appelle {\it dernier polaire} de $p$ par rapport à  
$\Sigma$.

Dans le cas où $\Sigma$ est la réunion des $H_i$,
le dernier polaire $p^\perp$ de $p$ est bien défini
et égal à $p^\rep$.

\subsection{Polarité par rapport à un convexe}

Étant donné un ouvert convexe propre $C$ de $\project(V)$,
on définit le polaire d'un hyperplan $H$ ne rencontrant pas $\bar C$
comme 
le centre de gravité de $C$ dans n'importe quelle carte affine envoyant 
$H$ à l'infini.
Comme le centre de gravité est une notion affine, ce point est bien défini.
Dans un espace projectif de dimension $n$, les hyperplans engendrés par les 
faces
d'un simplexe délimitent $n+2$ domaines convexes. Un hyperplan ne contenant 
aucun des
sommets du simplexe rencontre tous ces domaines sauf un. Le polaire de 
l'hyperplan par rapport
à ce convexe définit le polaire $H^\circ$ par rapport au simplexe.
Alors $H^\circ=H^\rep$ et on peut définir $p^\circ=p^{*\circ *}$.


\bibliographystyle{alpha}
\bibliography{biblio.bib}

\end{document}